\documentclass[12pt]{amsart}

\usepackage{amsmath} 
\usepackage{amsfonts}
\usepackage{amssymb}
\usepackage{amsthm}
\usepackage{mathtools}
\usepackage{enumerate}
\usepackage{setspace}
\usepackage{tikz}
\usepackage{graphicx}
\usetikzlibrary{matrix}

\newtheorem{definition}{Definition}[section]
\newtheorem{lemma}{Lemma}[section]
\newtheorem{thm}{Theorem}[section]
\newtheorem{prop}{Proposition}[section]
\newtheorem{coro}{Corollary}[section]
\newtheorem{remark}{Remark}[section]
\newtheorem{thm*}{Theorem}[section]
\numberwithin{equation}{section}

\newcounter{example}[section]
\newenvironment{example}[1][]{\refstepcounter{example}\par\medskip
   \noindent \textbf{Example~\theexample} (#1){\bf .}  \rmfamily}{\bigskip}
   
\newcommand{\pr}{\partial}
\newcommand{\veps}{\varepsilon}
\newcommand{\lm}[2]{\lim\limits_{#1\to #2}}
\newcommand{\ovr}[1]{\overline{#1}}   
\newcommand{\ddt}[2]{\dfrac{d #1}{d #2}}
\newcommand{\be}{\begin{equation}}
\newcommand{\ee}{\end{equation}}
\newcommand{\bee}{\begin{equation*}}
\newcommand{\eee}{\end{equation*}}
   
\def\tr{\textnormal{tr}}
\def\dv{\textnormal{div}}

\def\Lap{\Delta}
\def\grad{\nabla}

\def\dint{\displaystyle\int}
\def\R{\mathbb{R}}
\def\vol{\mathrm{vol}\,}
\def\la{\langle}
\def\ra{\rangle}
\def\S{\Sigma}
\def\({\left(}
\def\){\right)}
\def\a{\alpha}
\def\b{\beta}
\def\To{\longrightarrow}

\def\bH{\mathbf{H}}

\def\hess{\textnormal{Hess}}
\def\tf{\tilde{f}}
\def\Hb{\mathbb{H}}
\def\s{\sigma}
\def\N{\mathcal{N}}
\def\w{\omega}
\def\g{\gamma}
\def\V{\mathcal{V}}
\def\B{\mathcal{B}}
\def\tb{\tilde{b}}

\def\k{\kappa}
\def\bk{b_{\k}}

\def\Vk{V_{\k}}
\def\ek{e_{\k}}
\def\bO{\overline{\Omega}}
\def\bS{\mathbb{S}}

\def\graph{\textnormal{graph}}

\def\bM{\mathbf{M}}
\def\ADM{\textnormal{ADM}}
\def\m{\mathbf{m}}
\def\df{\grad^b f}
\def\dkf{\grad^{b_{\k}} f}
\def\defeq{\coloneqq}

\allowdisplaybreaks

\title[On the Stability of the PMT for AH Graphs]{On the Stability of the Positive Mass Theorem for Asymptotically Hyperbolic Graphs}
\author[{Cabrera Pacheco}]{Armando J. {Cabrera Pacheco}}
\address{Department of Mathematics, Universit\"at T\"ubingen,  72076 T\"{u}bingen, Germany.}
\email{cabrera@math.uni-tuebingen.de}

\begin{document}

\begin{abstract}
The positive mass theorem states that the total mass of a complete asymptotically flat manifold with non-negative scalar curvature is non-negative; moreover, the total mass equals zero if and only if the manifold is isometric to the Euclidean space. Huang and Lee  \cite{H-L} proved the stability of the positive mass theorem for a class of $n$-dimensional ($n \geq 3$) asymptotically flat graphs with non-negative scalar curvature, in the sense of currents. Motivated by their work and using results of Dahl, Gicquaud and Sakovich \cite{D-G-S}, we adapt their ideas to obtain a similar result regarding the stability of the positive mass theorem, in the sense of currents, for a class of $n$-dimensional $(n \geq 3)$ asymptotically hyperbolic graphs with scalar curvature bigger than or equal to $-n(n-1)$.
\end{abstract}

\maketitle

\section{Introduction}
\thispagestyle{empty}

Asymptotically hyperbolic manifolds arise in a natural way as spacelike slices of spacetimes satisfying the Einstein equations with negative cosmological constant $\Lambda$. Roughly speaking, a Riemannian manifold $(M^n,g)$ is asymptotically hyperbolic if the metric $g$ approaches the metric of the hyperbolic space $\Hb^n$ at infinity at an appropriate rate. In the asymptotically flat case (when $\Lambda =0$ and the metric $g$ approaches the Euclidean metric at infinity), there is a well-known notion of total mass, called ADM mass \cite{ADM, Bartnik}. The positive mass theorem proven by Schoen and Yau \cite{S-Y-PMT,S-Y-PMT2} and Witten \cite{W-PMT}, states that if $(M^n,g)$ is an asymptotically flat manifold with non-negative scalar curvature, then its total mass, $m_{\ADM}(g)$, is non-negative. In addition, it has a rigidity statement which says that $m_{\ADM}(g)=0$ if and only if $(M^n,g)$ is isometric to the Euclidean space of dimension $n$. 

The rigidity statement of the positive mass theorem for asymptotically flat manifolds leads naturally to a stability question: if $m_{\ADM}(g)$ is close to 0, is $(M^n,g)$ close to the Euclidean space of dimension $n$ is some sense? There has been various results in this direction by Allen \cite{Allen1}, Bray and Finster \cite{B-F}, Corvino \cite{Corvino}, Finster \cite{Finster}, Finster and Kath \cite{F-K}, Huang and Lee \cite{H-L}, Huang, Lee and Sormani \cite{H-L-S}, Lee \cite{Lee} and Lee and Sormani \cite{L-S}. We refer the reader to \cite{H-L-S} and \cite{S-S} for a brief description and short history of this problem.

The definition of total mass of an asymptotically hyperbolic manifold is more subtle and has been introduced by Wang \cite{Wang} and Chru\'sciel and Herzlich \cite{C-H}. Moreover, Wang in \cite{Wang} and Chru\'sciel and Herzlich in \cite{C-H} proved the positive mass theorem for asymptotically hyperbolic manifolds which are spin. Different versions of this theorem have been obtained by Andersson, Cai and Galloway in \cite{A-C-G} and, recently, by Chru\'sciel, Galloway, Nguyen and Paetz in \cite{C-G-N-P}, imposing conditions on the geometry at infinity. The stability problem in the asymptotically hyperbolic setting has been studied by Dahl, Gicquaud and Sakovich for conformally hyperbolic manifolds \cite{D-G-S-stability} adapting the ideas of Lee \cite{Lee}, and by Allen \cite{Allen2}, assuming that the manifolds can be foliated by a solution of the Inverse Mean Curvature Flow. Recently, Sakovich and Sormani proved the stability of the positive mass theorem for spherically symmetric asymptotically hyperbolic manifolds with respect to the intrinsic flat convergence \cite{S-S}, using similar methods to those developed by Lee and Sormani in \cite{L-S}. We also refer to \cite{S-S} for a very nice exposition of the history and description of results related to this problem.

In this work, we study the stability of the positive mass theorem for asymptotically hyperbolic graphs, adapting to this setting the ideas of Huang and Lee \cite{H-L}, and making use of results of Dahl, Gicquaud and Sakovich \cite{D-G-S}. We consider graphs over the hyperbolic space of dimension $n$, denoted by $\Hb^n$, inside $\Hb^{n+1}$. Here we use the model $\Hb^{n+1}=\R \times \Hb^n$ with a metric of the form $\bar{b}=V^2ds^2 + b$, where $b$ is the metric of $\Hb^n$, so that the set determined by setting $s$ equal to a constant is isometric to $\Hb^n$. An asymptotically hyperbolic graph is, in general terms, a continuous function $f:\Hb^n \setminus \Omega \To \R$ which is smooth on $\Hb^n \setminus \ovr{\Omega}$, where $\Omega \subset \Hb^n$ is open and has connected complement, and satisfies appropriate decay conditions (see Definition \ref{AHgraph-def}). If $\Omega \neq \emptyset$ we require $f$ to have a minimal boundary  (see Definition \ref{AHgraph-def}). 

Our main result is the following. 
 
\begin{thm}\label{main}
Let $M^i$ be a sequence of balanced asymptotically hyperbolic graphs  in $\Hb^{n+1}$ with scalar curvature $R \geq -n(n-1)$ ($n \geq 3$), and mean curvature vector pointing upward. Suppose that almost every level set of $M^i$ is star-shaped and outer-minimizing in the hyperbolic hyperplane. Let $\m^i$ denote the mass of $M^i$ and normalize the height so that the level set $M^i \cap \{s=0 \}$ has volume equal to  $\max \{ 2 \beta (\m^i)^{\frac{n-1}{n-2}}\w_{n-1},2 \beta \w_{n-1} \m^i \}$, where $\w_{n}$ denotes the volume of the standard unit round sphere $\bS^n$ and $\b>1$ is any fixed constant.  Then, if $\m^i \to 0$, then $\{ M^i \}$ converges to $\{ s = 0 \}$ in the sense of currents. 
\end{thm}

We remark that in contrast to the corresponding result in \cite{H-L}, we require that the mean curvature vector points upward as a hypothesis. The reason is that in the asymptotically flat case this condition follows from assuming that scalar curvature is non-negative \cite{H-W-RPI,H-L}, while in the asymptotically hyperbolic setting, to ensure that the mean curvature of an asymptotically hyperbolic graph with scalar curvature $R \geq -n(n-1)$ does not change sign, we would need impose that $\Omega$ is convex as in \cite[Proposition 5.1]{D-G-S}. 

As a consequence of the level sets being star-shaped, Theorem \ref{main} only works when the initial data set contains a horizon with a single connected component (or no horizon at all). However, this seems to be consistent with the existing result for the Riemannian Penrose inequality in this setting (see \cite{D-G-S,L-G}). The case when the horizon has several connected components will not be discussed here.

To obtain Theorem \ref{main} we adapt the methods developed by Huang and Lee in \cite{H-L} to the asymptotically hyperbolic setting, for which we use the results obtained by Dahl, Gicquaud and Sakovich in \cite{D-G-S}. In \cite{D-G-S}, the mass of asymptotically hyperbolic graphs (possibly with a minimal boundary) in $\Hb^{n+1}$ is written as the sum of the integral of $R(f) + n(n-1)$, where $R(f)$ denotes the scalar curvature of the graph of $f$ in $\Hb^{n+1}$, over the  ``exterior region'' and a boundary integral depending only on the geometry of the minimal boundary. Using this formula and arguing as Lam in \cite{Lam}, they obtained several Riemannian Penrose-like inequalities for asymptotically hyperbolic graphs in $\Hb^{n+1}$. 

In order to prove our main result, we follow the ideas of Huang and Lee in \cite{H-L}. Recall that in \cite{H-L} the expression of the ADM mass for an asymptotically flat graph in $\R^{n+1}$ as obtained by Lam \cite{Lam}, is applied to level sets of the graphing function to derive a differential inequality for the volume of the level sets. This inequality is then used to estimate the maximum height of the graph in terms of the ADM mass. After an appropriate normalization, it follows that a sequence of asymptotically hyperbolic graphs converges to a horizontal hyperplane in $\R^{n+1}$ (i.e., isometric to $\R^n$) in the sense of currents when their masses tend to zero.

It is worth to notice that Huang and Lee in \cite{H-L} had to treat differently the cases $n=3$ and $n=4$, mainly because of the use of the Minkowski inequality in $\R^n$. Interestingly, in our setting, using a Minkowski-like inequality obtained by Montiel and Ros in \cite{M-R} and used by Dahl, Gicquaud and Sakovich in \cite{D-G-S}, the argument works for any dimension $n \geq 3$.

We will use the notation in \cite{D-G-S} and \cite{H-L} as much as possible. In addition, the structure of the article is similar to \cite{H-L}. Namely, in Section \ref{mass-AH-graphs} we provide all the necessary definitions and terminology. In Section \ref{section-rescaling} we obtain results similar to those in \cite{D-G-S} in the case when the background metric of $\Hb^{n+1}$ is rescaled in a special way; this rescaling is used to obtain a differential inequality for the volume of the level sets as in \cite{H-L}. In Section \ref{volume-estimates} we obtain the desired differential inequality and use it to estimate the maximum height of asymptotically hyperbolic graphs. Finally, in Section \ref{convergence} we prove Theorem \ref{main}.

\medskip

{\bf Acknowledgements.} I would like to thank Lan-Hsuan Huang for pointing out this problem to me, for all her support and motivating discussions. I would also like to thank Carla Cederbaum, Kwok-Kun Kwong and Jason Ledwidge for very valuable comments and discussions, and to Anna Sakovich for a thorough reading of the first draft of this work and for her very helpful observations.  In addition, I would like to thank the Carl Zeiss Foundation for the generous support. This work was partially funded by NSF grant DMS 1452477.

\section{The mass of asymptotically hyperbolic graphs} \label{mass-AH-graphs}

We start with the definition and some properties of asymptotically hyperbolic graphs in the context of general relativity. Let $\Hb^n$ denote the hyperbolic space of dimension $n$ and $b$ its metric. In polar coordinates on $\R^n$ given by $(r,\theta)$, with $\theta \in \bS^{n-1}$, the metric $b$ takes the form
\be
b=dr^2 + \sinh^2 r  \s,
\ee
where $\s$ denotes the usual round metric on $\mathbb{S}^{n-1}$. Asymptotically hyperbolic manifolds are, roughly speaking, Riemannian manifolds whose metrics approach the metric $b$ at infinity. As in the asymptotically flat case, one is interested in a global invariant referred to as total mass (in the asymptotically flat case this notion is known as the ADM mass); however, in the asymptotically hyperbolic setting, this notion is more subtle and depends on the structure of the set $\mathcal{N} = \{ V \in C^{\infty}(\Hb^n) \, | \, \hess^b V = Vb \}$ (see \cite{C-H,Michel,D-G-S}). Note that $V_{(0)} := V_{(0)} (r)= \cosh r$  and $V_{(i)}:=x^i \sinh r$, where $x^i$ denote the coordinate functions on $\R^n$ restricted to $\bS^{n-1}$, give a basis of the vector space $\mathcal{N}$. Moreover, $\mathcal{N}$ can be endowed with a Lorentzian product $\eta$, in such a way that $V_{(0)}$ is future directed, $\eta(V_{(0)},V_{(0)}) = 1$ and $\eta(V_{(i)},V_{(i)}) =-1$ for $i=1,...,n$. For a precise discussion on how the structure of $\mathcal{N}$ affects the global invariant that corresponds to the total mass, see \cite{C-H}. This notion of total mass has been successfully adapted to the context of asymptotically hyperbolic graphs in \cite{D-G-S}. We will mostly follow \cite{D-G-S}  and summarize some definitions and notation needed for our purposes, but will often refer to \cite{D-G-S} (and the references therein) for details.

\begin{definition}[\cite{D-G-S}] \label{def-AH} 
A Riemannian manifold $(M,g)$ is said to be asymtptotically hyperbolic if there exist a compact subset $K$ and a chart at infinity $\Phi:M \setminus K \to \Hb^n \setminus B$, where $B$ is a closed ball in $\Hb^n$, for which:
\begin{enumerate}
\item $\Phi_*g$ and $b$ are uniformly equivalent, i.e., there is a constant $C>0$ such that
\bee
C^{ -1} b(X,X) \leq \Phi_*g(X,X) \leq C b(X,X),
\eee
\item 
\bee
\int_{\Hb^n \setminus B} (|e|^2 + |\grad^b e |^2) \cosh r_b \, d\mu^b < \infty,
\eee
\item 
\bee
\dint_{\Hb^n \setminus B} | R(g) + n(n-1)| \cosh r_b \, d\mu^b < \infty,
\eee
where $e=\Phi_* g - b$, $r_b$ is the hyperbolic distance from an arbitrary point in $\Hb^n$ and $R(g)$ is the scalar curvature of $g$.
\end{enumerate}
\end{definition}

\begin{definition}[\cite{D-G-S}]\label{AH-functional}
The mass functional of $(M,g)$, with respect to $\Phi$, is the functional defined on $\N$ by
\be\label{mass-functional}
H_{\Phi}(V)=\lm{r}{\infty} \int_{S_r} (V(\dv^b e - d \tr^b e) + (\tr^b e) dV - e(\grad^b V, \cdot))(\nu_r) d\mu^b,
\ee
where $S_r$ denotes the coordinate sphere of radius $r$ in $\Hb^n$, and $\nu_r$ the outward normal vector to $S_r$.
\end{definition}

Under the conditions of Definition \ref{def-AH} this limit exists and is finite. Moreover, if $A$ is an isometry of the hyperbolic space with respect to the metric $b$, then $H_{A \circ \Phi} (V) = H_{\Phi}(V \circ A^{ -1})$ (see \cite{C-H}). If $H_{\Phi}(V) > 0$ for all positive $V$ in $\mathcal{N}$, the mass of an asymptotically hyperbolic manifold $(M,g)$ is defined as follows (see \cite{D-G-S} for details).

\begin{definition}[\cite{D-G-S}]\label{AH-mass}
The mass of an asymptotically hyperbolic manifold $(M,g)$ is defined as
\be
\m = \dfrac{1}{2(n-1)\omega_{n-1}} \inf_{V \in \mathcal{N}^1} H_{\Phi}(V),
\ee
where $\mathcal{N}^1$ is the set of positive functions in $\mathcal{N}$ with $\eta(V,V)=1$. Moreover, one can replace $\Phi$ by $A \circ \Phi$ for a suitable isometry $A$, so that
\be
\m =  \dfrac{1}{2(n-1)\omega_{n-1}} H_{\Phi}( V_{(0)}).
\ee
Coordinates with this property are called balanced.
\end{definition}

We now turn our attention to graphs over $\Hb^n$ as submanifolds of $\Hb^{n+1} = (\R \times \Hb^n, \bar b)$. To describe the model space of $\Hb^{n+1}$, in polar coordinates on $\R^n$ as above, we fix the function $V:=V(r) = \cosh r$ and let
\be
\bar b = V(r)^2 ds \otimes ds + b.
\ee

Let $\Omega$ be a bounded open subset of $\Hb^n$ and $f: \Hb^n \setminus \Omega \longrightarrow \R$ be a continuous function, which is smooth on $\Hb^n \setminus \ovr{\Omega}$. Let $\Phi$ be the chart of the graph
\be
\graph[f] = \{  (s,x) \in \R \times \Hb^n  \ | \ f(x) = s  \},
\ee
defined as the inverse of the usual parametrization of the graph, $\Psi : \Hb^n \setminus \Omega \to \graph[f]$, given by $\Psi(x) = (f(x),x)$. We endow $\graph[f]$ with the metric induced by $\Hb^{n+1}$. Following \cite{D-G-S} and \cite{H-L}, we make the following definition.

\begin{definition}[Asymptotically hyperbolic graph] \label{AHgraph-def} 
Let $\Omega \subset \Hb^n$ ($n \geq 3$) be an open and bounded subset, such that its complement is connected. Consider a continuous function $f: \Hb^n \setminus \Omega \longrightarrow \R$, smooth on $\Hb^{n} \setminus \ovr{\Omega}$. Suppose that 
\begin{enumerate}[(i)]
\item $e = V^2 df \otimes df$ satisfies (2) and (3) in Definition \ref{def-AH}, and
\item $V^2 |\grad^b f|_{b}^2 \to 0$ at infinity. 
\end{enumerate}
Then we say that $\graph[f]$ (or simply $f$) is asymptotically hyperbolic with respect to the chart $\Phi$ as above. We say that  $\graph[f]$ (or $f$) is balanced if $\Phi$ is a set of balanced coordinates. In the case when $\Omega = \emptyset$, we say that $f$ is entire.
\end{definition}

We now set $c_n = 2(n-1) \w_{n-1}$, where $\w_{n-1}$ denotes the volume of the standard round metric on $\bS^{n-1}$.

\begin{remark}\label{AHgraph-mass}
By Definition \ref{AH-mass}, if $f$ is asymptotically hyperbolic and balanced, the mass of $f$ is given by $\m(f)=\frac{1}{c_n} H_{\Phi} (\cosh r )$.
\end{remark}

\begin{remark}\label{remark-asymptotics}
Note that by condition (i) in the previous definition, we can check that $\lim_{r \to \infty} f(r) = C$ for some constant $C$, in contrast to the asymptotically flat case in \cite{H-L} where an asymptotically flat graph is allowed to approach $\pm \infty$ as $|x| \to \infty$. In fact, the usual decay conditions to be imposed in the asymptotically flat case in order to ensure that the ADM mass is finite, namely, $|\grad f|=O(|x|^{\frac{q}{2}})$ for some $q > \frac{n-2}{2}$, imply $\lim_{|x| \to \infty} f(x) = C$ for dimensions $n \geq 6$ (see \cite{H-W-RPI}).
\end{remark}

We now combine Definition \ref{AH-functional} and Remark \ref{AHgraph-mass} and give the definition of mass of a balanced asymptotically hyperbolic function.

\begin{definition}[Mass of an asymptotically hyperbolic graph] \label{mass-AHgraph}
If $f$ is an asymptotically hyperbolic and balanced, its mass is given by
\be
\m(f) = \dfrac{1}{c_n} \lm{r}{\infty} \int_{S_r} (V(\dv^b e - d \tr^b e) + (\tr^b e) dV - e(\grad^b V, \cdot))(\nu_r) d\mu^b,
\ee
where $e = V^2 df \otimes df$, $V := V(r) =\cosh r$, $S_r$ is the coordinate sphere of radius $r$ in $\Hb^{n}$, and $\nu_r$ the outward normal vector to $S_r$.
\end{definition}

\begin{remark}
The positive mass theorem for asymptotically hyperbolic manifolds with scalar  curvature greater or equal to $-n(n-1)$ has been proven for spin manifolds \cite{C-H} and for lower dimensions with special asymptotic behavior \cite{A-C-G,C-G-N-P}. Note that by Corollary \ref{kmass-est2} below, all asymptotically flat graphs considered here will have non-negative mass.
\end{remark}

\begin{definition}
We say that an asymptotically hyperbolic function $f:\Hb\setminus \Omega \To \R$ has a minimal boundary if $\pr \Omega \neq \emptyset$ and $f$  is constant on each component of $\pr \Omega$ and $|\df|_b \to \infty$ on $\pr \Omega$.
\end{definition}

\begin{example}[AdS-Schwarzschild graphs]
One of the most important examples of asymptotically hyperbolic manifolds are the AdS-Schwarzschild manifolds. These manifolds represent initial data sets for the Einstein Equations with negative cosmological constant, corresponding to the gravitational field surrounding a non-rotating spherical body or a black hole. For $n \geq 3$ and $m \geq 0$, let $\rho_0=\rho(m)$ be the largest solution of
\bee
1 + \rho^2 - \dfrac{2m}{\rho^{n-2}} = 0.
\eee
The AdS-Schwarzschild manifold of mass $\m = m$ is the manifold $(\rho_0,\infty) \times \bS^2$ with the metric
\bee
g_m = \dfrac{1}{1+ \rho^2 - \frac{2m}{\rho^{n-2}}} d\rho^2 + \rho^2 \s.
\eee
Clearly, if $m=0$, $g_0$ is the metric of $\Hb^n$. In order to realize a mass $m$ AdS-Schwarzschild manifold as an asymptotically hyperbolic graph it is useful to set $\rho = \sinh  r$ so that the metric $b$ on $\Hb^n$ takes the form
\bee
b = \dfrac{d\rho^2}{1+\rho^2}+\rho^2 \s,
\eee
and the metric $\bar b$ on $\Hb^{n+1}$ is then given by $\bar b = V^2 ds^2 + b$,
with $V:=V(\rho)=\sqrt{1+\rho^2}$ . Therefore, the \emph{AdS-Schwarzschild graph of mass $m$} is given by
\bee
f_{m}(\rho) = \dint_{\rho_0}^{\rho} \dfrac{1}{\sqrt{1+s^2}}\sqrt{\dfrac{1}{1+s^2-\frac{2m}{s^{n-2}}} - \dfrac{1}{1+s^2}}\,ds.
\eee
Note that $\grad f_{m}(\rho) \to \infty$ when $\rho \to \rho_0$, that is, $f_{m}$ has a minimal boundary.
\end{example}

Dahl, Gicquaud and Sakovich obtained Penrose type inequalities for asymptotically hyperbolic graphs with minimal boundary in \cite{D-G-S}, including the following one.
\begin{thm}[\cite{D-G-S}]
Let $f:\Hb^n \setminus \Omega \To \R$ be a balanced asymptotically hyperbolic graph with minimal boundary and scalar curvature $R(f) \geq -n(n-1)$ in $\Hb^{n+1}$. Suppose that $\pr \Omega$ is mean-convex ($H \geq 0$) and that contains an inner ball centered at the origin of radius $r_0$. Then,
\be\label{RPI-AH}
\m(f) \geq \dfrac{1}{2 \w_{n-1}} V(r_0) |\pr \Omega|_b,
\ee
where $V=V(r)=\cosh(r)$.
\end{thm}
This lower bound will suffice for our purposes, although it is not optimal. We remark that in \cite{GWW}, Ge, Wang and Wu defined a higher order mass invariant for asymptotically hyperbolic graphs and obtained corresponding Penrose type inequalities.

Now, we would like to impose some conditions on the ``shape'' of the graph. In $(\Hb^{n+1},\bar b)$, $n_0 := (1,\vec{0})$ is a natural vector to consider in the normal space to $T_p \Hb^n \subset T_p\Hb^{n+1}$, for $p \in \Hb^{n+1}$. We denote the mean curvature and the mean curvature vector of $\graph[f]$ in $\Hb^{n+1}$ by $\ovr{H}$ and $\ovr{\bH}$, respectively. The convention we use for the mean curvature is that the unit sphere in $\R^n$ has positive mean curvature equal to $n-1$ with respect to the inward pointing normal vector.

\begin{definition}
We say that the mean curvature vector of $\graph[f]$ in $\Hb^{n+1}$, $\ovr{\bH}$, is \emph{pointing upward} if at every point on $\graph[f]$, either $\bar{b}(\ovr{\bH},n_0) > 0$ or $\ovr{{\bH}}=0$. 
\end{definition}

If $h$ is a regular value of a smooth function $f:\Hb^{n}\setminus \Omega \To \R$, we denote the mean curvature and the mean curvature vector of the level set $\S_h$ in $\Hb^n$ by $H_{\S_h}$ and $\bH_{\S_h}$, respectively.  The following result \cite[Proposition 5.4]{D-G-S}, relates $\bH$ and $\bH_{\S_h}$, when $h$ is a regular value of $f$.

\begin{prop}[\cite{D-G-S}]\label{mean-curv-comparison}
Let $f:\Hb^{n}\setminus \Omega \To \R$ be a smooth function, and let $s_0$ be given. Assume that $s_0$ is a regular value for $f$, and let $\S:=\graph[f] \cap \{ s=s_0 \}$. Let $\nu$ be a unit normal vector field to $\graph[f]$ in $\Hb^{n+1}$, $\eta$ be a normal unit vector field to $\S$ in $\{ s_0 \} \times \Hb^n$, and $H_{s_0}$ denote the mean curvature of $\S$ in $\{ s_0 \} \times \Hb^n$, with respect to $\eta$. Then
\bee
\bar b (\nu,\eta) \ovr{H}H_{s_0} \geq \dfrac{R(f) + n(n-1)}{2} + \dfrac{n}{2(n-1)}\bar b(\nu,\eta)^2H_{s_0}^2.
\eee
\end{prop}

The following lemma follows from Proposition \ref{mean-curv-comparison}, and corresponds to Corollary 2.11 in \cite{H-L}.

\begin{lemma} \label{level-set-mean-convex}
Let $f:\Hb^n \setminus \bO \longrightarrow \R$ be a balanced asymptotically hyperbolic graph, with scalar curvature satisfying $R(f) \geq -n(n-1)$ and upward pointing mean curvature vector field. Let $\S_h$ be a level set of a regular value $h$, then
\be
\bH_{\S_h} = - H_{\S_h}\dfrac{\grad^b f}{|\grad^b f|_b} \quad \text{and} \quad
H_{\S_h} \geq 0.
\ee
\end{lemma}
\begin{proof}
Using Proposition \ref{mean-curv-comparison} with $\S=\{ h \} \times \S_h$, it follows that $\bar b(\overline{\bH} , \bH_{\S_h}) \geq 0$, where $\overline{\bH}$ is the mean curvature vector of $\graph[f]$ in $\Hb^{n+1}$ and $\bH_{\S_h}$ the mean curvature vector in $\Hb^n$ (here we are abusing notation by using the same name for the mean curvature vector of $\S_h$ and $\{ h \} \times \S_h$ in $\Hb^n$ and $\{ h \} \times \Hb^n$, respectively). The condition that the mean curvature vector is pointing upward implies that
\bee
\overline{\bH} =\overline{H}\dfrac{(V^{-2},-\grad^b f)}{\sqrt{V^{-2}+|\grad^b f|_b^2}}, \ \text{with} \ \overline{H} \geq 0.
\eee
On the other hand, we have
\bee
\bH_{\S_h} = H_{\S_h}\dfrac{(0,-\grad^b f)}{|\grad^b f|_b}.
\eee
It follows that
\begin{align*}
0 \leq b(\overline{\bH} , \bH_{\S_h}) = \overline{H}H_{\S_h}\dfrac{|\grad^b f|_b}{\sqrt{V^{-2}+|\grad^b f|_b^2}},
\end{align*}
and hence, $H_{\S_h} \geq 0$.
\end{proof}

Let $f:\Hb^{n} \setminus \Omega \To \R$ be an asymptotically hyperbolic function. We will be interested in estimating the volume of the level sets, which might not be smooth. Recall that if $f$ has a minimal boundary, then $f$ is locally constant on $\pr \Omega$ and hence we can define a function $\bar f$ with domain $\Hb^n$ (and coinciding with $f$ on $\Hb^n \setminus \Omega$), by extending $f$ to be constant on each connected component of $\ovr{\Omega}$.  If $f$ is entire, then $\bar f=f$.  In addition, for each $h \in \R$ define the set $\Omega_h := \{ x \in \Hb^n \, : \, \bar f(x) < h \}$ and $\S_h = \pr^* \Omega_h$, which denotes the reduced boundary of $\Omega_h$, see \cite{Simon} and \cite{E-G} for relevant definitions. Notice that by Sard's theorem, the set of critical points of $f$ has measure zero. On the other hand, for each regular value $h$, $\S_h$ is a smooth level set. We define the volume along the level sets as
\bee
\V(h) = |\S_h|_b = |\pr^* \Omega_h|_b,
\eee
where $|\cdot|_b$ denotes the $(n-1)$-dimensional Hausdorff measure in  $\Hb^n$. From the lower semicontinuity of perimeter, it follows that the function $\V$ is lower semicontinuous.

Moreover if $E \subset \Hb^n$ is bounded and has finite perimeter, we say that $\pr^*E$ is outer-minimizing if $|\pr^*E|_b \leq |\pr^* F|_b$ for any bounded set $F \subset \Hb^n$ that contains $E$.

In the following lemma we rule out the existence of a maximum in the interior or in the boundary of the domain of an asymptotically hyperbolic function. The proof follows identically as in \cite[Lemma 3.3]{H-L} and it is included here for the reader's convenience.

\begin{lemma} \label{lemma-height}
Let $f:\Hb^n \setminus \Omega \To \R$ be a non-constant smooth asymptotically hyperbolic function, either entire or with minimal boundary, with upward pointing mean curvature vector field. Let $h_{\max} = \lim_{r \to \infty} f(r,\theta)$, which is a real number. Then $f < h_{\max}$ everywhere and $\V(h)$ is finite for all $h < h_{\max}$. If in addition, $\S_h$ is outer-minimizing for almost every $h$, then $\V$ is non-decreasing on $(-\infty,h_{\max})$. 
\end{lemma}

\begin{proof}
From Section 3 in \cite{D-G-S}, given any set of coordinates in $\Hb^n$ (for which we will use latin indices), and since the mean curvature vector of the graph points upward, we have
\bee
\ovr{H} = \dfrac{V^2}{1+V^2|\grad^b f|^2}\( b^{ij} - \dfrac{V^2 f^i f^j}{1+V^2|\grad^b f|^2} \)\left[ \nabla^2_{ij}f + \dfrac{f_i V_j + V_i f_j}{V} + Vb(df,dV)f_if_j \right] \geq 0.
\eee
By the strong maximum principle, a maximum can not be achieved at an interior point unless $f$ is a constant. If $f$ has a minimal boundary, note that if the maximum is achieved on $\pr \Omega$ it will contradict the assumption that the mean curvature vector points upward. Hence 
\bee
f(x) < \sup\limits_{\Hb^n\setminus\Omega} f = \lim_{r \to \infty} f(r,\theta)=h_{\max},
\eee
for all $x \in \Hb^n \setminus \Omega$. It follows that for any $h < h_{max}$, $\Omega_h$ is a bounded subset and thus $\V(h)$ is finite.

Suppose now that $\S_h$ is outer-minimizing for almost every $h$. Let $h_1 < h_2 < h_{\max}$ and $\veps > 0$. Using the left lower semicontinuity of $\V$ and since the outer-minimizing level sets are dense by assumption, we can find $h < h_1 < h_2$ such that
\bee
\V(h_1) \leq \V(h) + \veps \leq \V(h_2) + \veps.
\eee
Since $\veps$ is arbitrary, it follows that $\V$ is non-decreasing.
\end{proof}

\section{Rescaling of asymptotically hyperbolic graphs}\label{section-rescaling}

In order to prove the stability of the positive mass theorem following the procedure in \cite{H-L}, we need to be able to rescale the asymptotically hyperbolic graphs appropriately. We start by defining the hyperbolic space of radius $\k^{-1} > 0$, denoted by $\Hb^n_{\k}$, as $(\R^n,b_{\k})$ with the metric given (in polar coordinates in $\R^n$) by
\be
b_{\k} \defeq dr^2 + \dfrac{\sinh^2(\k r)}{\k^2} \s,
\ee
which clearly has scalar curvature $R(b_{\k}) = -\k^2 n(n-1)$. In a similar fashion as in the definition of $\Hb^{n+1}$, we consider the model of $\Hb^{n+1}_{\k}$ as $(\R \times \Hb^n_{\k},\bar b_{\k})$, with
\be
\bar b_{\k} \defeq \cosh^2( \k r) ds^2 + b_{\k}.
\ee
A straightforward computation gives that the scalar curvature of $\bar b_{\k}$ is given  by $R(\bar b_{\k})= - \k^2 (n+1)n$. Notice that when $\k \to 0$, $\bar b_{\k}$ approaches the Euclidean metric on $\R^{n+1}$ and when $\k = 1$, $\bar b_{\k} = \bar b$.

\begin{definition}[$\k$-mass]
Let $V_{\k} :=V_{\k}(r) = \cosh(\k r)$ and $e_{\k} := V_{\k}^2 df \otimes df$. Suppose that $f:\Hb^n \setminus \Omega \To \R$, where $\Omega$ is open and bounded with connected complement. We define the $\k$-mass of $f$ as the quantity
\be
\m_{\k}(f)\defeq \dfrac{1}{c_n}  \lm{r}{\infty} \int_{S_r} (V_{\k}(\dv^{b_{\k}} e_{\k} - d \tr^{b_{\k}} e_{\k}) + (\tr^{b_{\k}} e_{\k}) dV_{\k} - e_{\k}(\grad^{b_{\k}} V_{\k}, \cdot))(\nu_r) d\mu^{b_{\k}}.
\ee
\end{definition}

\begin{remark}
The name $\k$-mass is motivated by the fact that when $\k=1$ and $f$ is an asymptotically hyperbolic function which is balanced, then $\m_{1}(f)=\m(f)$, that is, it is the mass of an asymptotically hyperbolic graph as in Definition \ref{mass-AHgraph}. For the $\k$-mass to be meaningful as a mass invariant, one would need to modify the definition of the mass of an asymptotically hyperbolic manifold accordingly with an appropriate  scaling of the cosmological constant, and then define the corresponding asymptotically hyperbolic graphs. We do not require such modification for our goals.
\end{remark}

From now on, we set $V_{\k} := V_{\k}(r) = \cosh(\k r)$ and $e_{\k} := V_{\k}^2 df \otimes df$. To evaluate $\m_{\k}$ we note that, from the computations in \cite[Section 3]{D-G-S}, it follows that
\be \label{VR-div}
\mathcal{R}_{f,\k} V_{\k} =  \dv^{b_{\k}} \left[ \dfrac{1}{1 + \Vk^2| \dkf|^2_{\bk}}\(\Vk \dv^{\bk}\ek  - \Vk d\tr^{\bk}\ek - \ek(\grad^{\bk}\Vk,\cdot) + (\tr^{\bk} \ek) d\Vk \) \right],
\ee
where $\mathcal{R}_{f,\k} = R_{\k}(f) + \k^2 n(n-1)$, and $R_{\k}(f)$ denotes the scalar curvature of $\graph[f]$ as a submanifold of $\Hb^{n+1}_{\k}$. The following lemma is similar to Lemma 3.2 in \cite{D-G-S} and gives a way to express the quantity $\m_{\k}$ in terms of a level set $\S_h$. Recall that $\Omega_h = \{ x \in \Hb^n , : \, \bar f < h \}$, where $\bar f$ is the extension of $f$ to $\Hb^n$ obtained by setting $\bar f$ constant on each connected component of $\ovr{\Omega}$ (See Section \ref{mass-AH-graphs}).

\begin{lemma} \label{k-mass}
Let $f$ be an asymptotically hyperbolic graph. If $\S_h$ is the level set of a regular value $h$, then
\be\label{k-mass-expression}
c_n \m_{\k}(f) =\dint_{\Hb^n_{\k} \setminus \Omega_h}  ( R_{\k}(f) + \k^2 n(n-1))\Vk\, d\mu^{\bk}   + \dint_{\S_h} \Vk H_{\k} \dfrac{\Vk^2|\grad^{\bk} f|_{\bk}^2}{1 + \Vk^2|\dkf|^2_{\bk}} \, d\S_h^{\bk},
\ee
with $c_n = 2(n-1) \w_{n-1}$. Here $H_{\k}$ denotes the mean curvature of $\S_h$ with respect to $b_{\k}$.
\end{lemma}

\begin{proof}

Let $\nu_r= \pr_r$ and $\nu = \frac{\grad^{\bk} f}{|\grad^{\bk} f|_{\bk}}$. We assume without loss of generality that $\nu$ is the outward pointing normal vector of $\S_h$ (note that this is the case if $f$ has mean curvature vector pointing upward).

\begin{align*}
&\dint_{\Hb^n_{\k} \setminus \Omega_h}  ( R_{\k}(f) + \k^2 n(n-1)) \Vk \, d\mu^{\bk} \\
=& \lm{r}{\infty} \dint_{B_r(0) \setminus \Omega_h} ( R_{\k}(f) + \k^2 n(n-1))  \Vk  \, d\mu^{\bk} \\
=& \lm{r}{\infty} \dint_{S_r(0)} \dfrac{1}{1 + \Vk^2|\dkf|^2_{\bk}}\(\Vk \dv^{\bk}\ek  - \Vk d\tr^{\bk}\ek - \ek(\grad^{\bk}\Vk,\cdot) + (\tr^{\bk} \ek) d\Vk \)(\nu_r) d\mu^{\bk} \\
&\quad - \dint_{\S_h} \dfrac{1}{1 + \Vk^2|\dkf|^2_{\bk}}\(\Vk \dv^{\bk}\ek  - \Vk d\tr^{\bk}\ek - \ek(\grad^{\bk}\Vk,\cdot) + (\tr^{\bk} \ek) d\Vk \)(\nu) d\S_h^{\bk}  \\
\begin{split}
=&c_n \m_{\k}(f)  \\
&\qquad - \dint_{\S_h} \dfrac{1}{1 + \Vk^2|\dkf|^2_{\bk}}\(\Vk \dv^{\bk}\ek  - \Vk d\tr^{\bk}\ek - \ek(\grad^{\bk}\Vk,\cdot) + (\tr^{\bk} \ek) d\Vk \)(\nu) d\S_h^{\bk},
\end{split}
\end{align*}
where we have used \eqref{VR-div} and the divergence theorem. 

On the other hand, note that 
\begin{align*}
\( - \ek(\grad^{\bk}\Vk,\cdot) + (\tr^{\bk} \ek) d\Vk \)(\nu) &= - \Vk^2 df(\grad^{\bk} \Vk) df(\nu) + \Vk^2 \bk^{ij} f_if_j d\Vk(\nu) \\
&=-\Vk^2f^i (\Vk)_i |\grad^{\bk} f|_{\bk} + \Vk^2 f^i(\Vk)_i |\grad^{\bk} f|_{\bk}  \\
&=0,
\end{align*}
and
\begin{align*}
\Vk \dv^{\bk} \ek (\nu) - \Vk d (\tr^{\bk}\ek) (\nu) &= \Vk^3 \Lap_{\bk} f df(\nu) - \Vk^3 \hess_{\bk} f(\grad^{\bk} f, \nu).
\end{align*}
Using that $f$ is constant on $\S_h$ and the relation
\be
\Lap_{\bk} f = \Lap_{\S_h} f + \hess_{\bk} f (\nu,\nu) + H_{\k} df(\nu),
\ee
we have
\begin{align*}
\Vk \dv^{\bk} \ek (\nu) - \Vk d (\tr^{\bk}\ek) (\nu)  &= \Vk^3 \Lap_{\bk} f |\grad^{\bk} f|_{\bk} - \Vk^3 (\Lap_{\bk} f  -H_{\k} df(\nu))|\grad^{\bk} f|_{\bk} \\
&= \Vk^3 H_{\k} |\grad^{\bk} f|_{\bk}^2.
\end{align*}
Therefore, it follows that
\be
\dint_{\Hb^n_{\k} \setminus \S_h} ( R_{\k}(f) + \k^2 n(n-1))  \Vk \, d\mu^{\bk} = c_n \m_{\k} - \dint_{\S_h} \Vk H_{\k} \dfrac{\Vk^2|\grad^{\bk} f|_{\bk}^2}{1 + \Vk^2|\dkf|^2_{\bk}} \, d\S_h^{\bk}.
\ee
\end{proof}

For the sake of completeness, we now state the positive mass theorem for asymptotically hyperbolic graphs as a direct consequence of the previous theorem and our assumptions.

\begin{coro}[Positive Mass Theorem for Asymptotically Hyperbolic Graphs] \label{kmass-est2}
Suppose that $f$ is an asymptotically hyperbolic graph with $\mathcal{R}_{f,\k} \geq 0$. Let $h$ be a regular value of $f$, then
\be \label{mass-est-level-set}  
 \m_{\k}(f) \geq \dfrac{1}{c_n} \dint_{\S_h} \Vk H_{\k} \dfrac{\Vk^2|\grad^{\bk} f|_{\bk}^2}{1 + \Vk^2|\dkf|^2_{\bk}} \, d\S_h^{\bk}.
\ee
In particular, if $f$ is a balanced asymptotically hyperbolic graph and $\S_h$ is a mean-convex level set, then 
\bee
\m(f)=\m_{1}(f) \geq 0.
\eee
That is, the mass of such a graph is non-negative.
\end{coro}

In addition, we obtain the rigidity of the positive mass theorem for the class of graphs that we are considering.
\begin{coro}[Rigidity of the Positive Mass Theorem for Asymptotically Hyperbolic Graphs]
Suppose that $f$ is a balanced asymptotically hyperbolic graph with $R(f) \geq -n(n-1)$ such that the level sets $\S_h$ are mean-convex for almost every $h$, if $f$ is non-constant. Then $\m(f)=0$ if and only if $f$ is constant and hence its graph is isometric to $\Hb^n$.
\end{coro}
\begin{proof}
Suppose that $f$ is non-constant. Let $h$ be a real number such that $\S_h$ is a smooth level set. From \eqref{mass-est-level-set}, if $\m(f)=0$ then $H_{\S} \equiv 0$, which is a contradiction since there are no closed minimal hypersurfaces in $\Hb^n$ \cite{Myers}. Thus $f$ has to be constant and hence $\graph[f]$ is isometric to $\Hb^n$. On the other hand, if $f$ is constant, then its graph is isometric to $\Hb^n$ so $R(f)=-n(n-1)$. It follows that $\m(f)=0$ directly from the definition of mass (see \eqref{mass-AHgraph}).
\end{proof}

\section{Volume estimates} \label{volume-estimates}

The main goal of this section is to obtain a differential inequality for the volume function $\V$ of the level sets of an asymptotically hyperbolic function $f$, which ultimately leads to an estimate of $\sup(f)$ in terms of its mass $\m(f)$. We will follow the procedure in \cite{H-L}. Note that the rescaling used in \cite{H-L} to make the mass of a given asymptotically flat function $f$ equal to 1 does not work here. This explains why the rescaled metric $b_{\k}$ and the $\k$-mass were introduced in Section \ref{section-rescaling}.

\begin{lemma} \label{vol-est-alpha}
Suppose that $f$ is a balanced asymptotically hyperbolic function, with $\mathcal{R}_{f,\k} \geq 0$ and upward pointing mean curvature vector field. Let $h$ be a regular value of $f$. Then for any real number $\a > 0$, we have
\begin{align*}
\ddt{}{h} \V_{\k}(h) \geq \dfrac{1}{\a} \left[ \dint_{\S_h} H_{\k} \Vk  \, d\S_h^{\k}  -(1+ \a^{ -2} ) c_n \m_{\k} \right]. 
\end{align*}
\end{lemma}

\begin{proof}
Let $\S_h$ be the level set  $\{ f = h \}$, which is smooth by hypothesis. As in \cite{H-L}, we are interested in studying the volume growth along level sets. Let us consider a variation along the level sets of $f$, $X: \S_h \times (- \veps , \veps) \longrightarrow \Hb^n_{\k}$, determined by
\bee
\pr_t X= \dfrac{\grad^{\bk} f}{|\grad^{\bk} f|^2_{\bk}},
\eee
with $X(\S_h, t =0) = \S_h$. Next, compute the first variation of $\V_{\k}$ with respect to  $X$ to obtain
\begin{align*}
\ddt{}{h} \V_{\k}(h) &= \dint_{\S_h} \dfrac{H_{\k}}{|\grad^{\bk} f|_{\bk}} \, d\S_h^{\k}.
\end{align*}
To estimate this quantity, recall that from Corollary \ref{kmass-est2}, if $\mathcal{R}_{f,\k}=R_{\k}(f) + \k^2 n(n-1) \geq 0$, then for the level set $\S_h$, we have
\be
c_n \m_{\k} \geq \dint_{\S_h} H_{\k}\Vk \dfrac{\Vk^2 |\grad^{\bk} f|^2_{\bk}}{1 + \Vk^2 |\grad^{\bk} f|^2_{\bk}} \, d\S_h^{\k},
\ee
where $c_n= 2(n-1)\w_{n-1}$.

For $\a \in [0,\infty)$, we have 
\begin{align*}
c_n \m_{\k} &\geq \dint_{\S_h} H_{\k}\Vk \dfrac{\Vk^2 |\grad^{\bk} f|^2_{\bk}}{1 + \Vk^2 |\grad^{\bk} f|^2_{\bk}} \, d\S_h^{\k} \\
& \geq \dfrac{\a^2}{\a^2 +1 }  \dint_{\S_h \cap  \{ \Vk |\grad^{\bk} f| \geq \a \} } H_{\k}\Vk  \, d\S_h^{\k},
\end{align*}
in particular,
\be
- \dint_{\S_h \cap  \{ \Vk |\grad^{\bk} f| \geq \a \} } H_{\k}\Vk  \, d\S_h^{\k} \geq -(1+ \a^{ -2} ) c_n \m_{\k}.
\ee
Then we estimate the rate of growth of the volume as follows
\begin{align*}
\ddt{}{h} \V_{\k}(h) &= \dint_{\S_h} \dfrac{H_{\k}}{|\grad^{\bk} f|_{\bk}} \, d\S_h^{\k} \\
&=  \dint_{\S_h \cap \{ \Vk |\grad^{\bk} f|_{\bk} < \a \} } \dfrac{H_{\k}}{|\grad^{\bk} f|_{\bk}} \, d\S_h^{\k} +  \dint_{\S_h  \cap  \{ \Vk |\grad^{\bk} f|_{\bk} \geq \a \}} \dfrac{H_{\k}}{|\grad^{\bk} f|_{\bk}} \, d\S_h^{\k}  \\
& \geq \dfrac{1}{\a} \dint_{\S_h \cap \{ \Vk |\grad^{\bk} f|_{\bk} < \a \} } \Vk H_{\k} \, d\S_h^{\k} \\
&=\dfrac{1}{\a} \left[ \dint_{\S_h} H_{\k}\Vk  \, d\S_h^{\k} -  \dint_{\S_h  \cap  \{ \Vk |\grad^{\bk} f|_{\bk} \geq \a \}} \Vk H_{\k} \, d\S_h^{\k} \right] \\
&\geq \dfrac{1}{\a} \left[ \dint_{\S_h} H_{\k} \Vk  \, d\S_h^{\k}  -(1+ \a^{ -2} ) c_n \m_{\k} \right]. 
\end{align*}
\end{proof}

\begin{lemma} \label{vol-k-estimate}
Suppose that $f$ is an asymptotically hyperbolic function with $\mathcal{R}_{f,\k} \geq 0$ and upward pointing mean curvature vector field. Let $h$ be a regular value of $f$ such that $\V_{\k}(h) > \frac{c_n}{n-1} \m_{\k}$. Then,
\be \label{vol-est}
\ddt{}{h} \V_{\k}(h) \geq c_n\dfrac{2\m_{\k}}{3\sqrt{3}}\( \dfrac{1}{2 \w_{n-1}\m_{\k}}\V_{\k}(h) -1 \)^{3/2} .
\ee
\end{lemma}
\begin{remark}
Note that in \cite{H-L} the level set $\S_h$ was assumed to be outer-minimizing. Even though we also assume the outer-minimizing condition in our main result, for this lemma it is not required as we do not rely in the Minkowski inequality for $\R^n$.
\end{remark}
\begin{proof}
Recall that from Lemma \ref{level-set-mean-convex}, the level set $\S_h$ is mean-convex. From the Minkowski-like inequality in the hyperbolic space proved in \cite{M-R} (c.f. \cite{D-G-S}), we have
\be
\dfrac{1}{c_n}\dint_{\S_h} H_{\k}\Vk  \, d\S_h^{\k} \geq \dfrac{1}{2\w_{n-1}}\V_{\k}(h).
\ee
Together with Lemma \ref{vol-est-alpha}, this gives
\be \label{vol-est1}
\ddt{}{h} \V_{\k}(h) \geq  \dfrac{c_n}{\a} \left[ \dfrac{1}{2 \w_{n-1}}\V_{\k}(h) -(1+ \a^{ -2} )\m_{\k} \right],
\ee
for each $\a \in [0,\infty)$. To simplify notation, set $\B = \frac{1}{2 \w_ {n-1}}\V_{\k}(h)$ and consider the right hand side of \eqref{vol-est1} as the function depending on $\a$ given by $g(\a)= \frac{c_n}{\a} \left[ \B  -(1+ \a^{ -2} )\m_{\k} \right]$. Thus
\begin{align*}
g'(\a) &= -\dfrac{c_n}{\a^2}\left[ \B  -(1+ \a^{ -2} )\m_{\k} \right] + \dfrac{c_n}{\a}\left[ \dfrac{2}{\a^3} \m_{\k} \right],
\end{align*}
and the only critical point of $g$ is located at 
\be
\a=\sqrt{3}\( \dfrac{\B}{\m_{\k}} -1 \)^{-1/2},
\ee
which is positive since $\B > \m_{\k}$ by our assumption. Moreover, $g$ attains its maximum value at this point, given  by
\begin{align*}
\dfrac{1}{\a} \left[ \B  -(1+ \a^{ -2} )\m_{\k} \right] &
=\dfrac{2\m_{\k}}{3\sqrt{3}}\( \dfrac{\B}{\m_{\k}} -1 \)^{3/2}.
\end{align*}
Hence, from \eqref{vol-est1} we obtain  
\be \label{vol-est2}
\ddt{}{h} \V_{\k}(h) \geq c_n\dfrac{2\m_{\k}}{3\sqrt{3}}\( \dfrac{1}{2 \w_{n-1} \m_{\k} }\V_{\k}(h) -1 \)^{3/2} .
\ee
\end{proof}

\begin{definition}
Let $f$ be a given balanced asymptotically hyperbolic function, we define the height $h_0$ as
\be \label{h0}
h_0 \defeq \sup\{ h \, : \, \V_1(h) \leq \sup\{ 2 \beta \m^{\frac{n-1}{n-2}}\w_{n-1},2 \beta \w_{n-1} \m \} \}, 
\ee
where $\beta>1$ is any fixed arbitrary constant and $\m=\m(f)$. In case the above set is empty, set $h_0 \defeq \min(f)$.
\end{definition}

Note that, if $\m < 1$, we have $\sup\{ 2 \beta \m^{\frac{n-1}{n-2}}\w_{n-1},2 \beta \w_{n-1} \m \} = 2 \beta \w_{n-1} \m$, which agrees with the definition of the corresponding value $h_0$ in \cite{H-L}.

\begin{remark}
As in \cite{H-L}, the choice of $\beta > 1$ has to be made to ensure that the constant $C$ appearing in Lemma \ref{f-h0-estimate} is finite. Note that although $C$ also depends on $\b$, by fixing it (as we do), we can consider the constant $C$ to be depending only on $n$, as the explicit value of $\b$ does not affect the results.
\end{remark}

\subsection{Rescaling}

Given a balanced asymptotically hyperbolic graph $f$ of mass $\m = \m(f) $, we can perfom a rescaling of the metric and the function to normalize it to have $\k$-mass $\m_{\k}(f)$ equal to 1, for $\k=\m^{\frac{1}{n-2}}$. More precisely, given an asymptotically hyperbolic function $f:\Hb^n\setminus\Omega \to \R$ with $R(f)  \geq -n(n-1)$, define the function (cf. \cite{H-L}):
\be
\tf(r,\theta) = \k^{-1}(f(\k r,\theta) - h_0),
\ee
for $\k > 0$ and $h_0$ as before.

The next lemma provides some formulas that will be used in what follows. The proof consists of straightforward computations.
\begin{lemma}  \label{formulae}
Let $f:\Hb^{n} \setminus \ovr{\Omega} \To \R$ smooth. Define $\tf$ as above. Denote by $\la \cdot,\cdot \ra_{\k}$ and $\la \cdot, \cdot \ra$ the product with respect to $b_{\k}$ and $b$, respectively. Then
\begin{enumerate}[(i)]
\item $| \grad^{\bk} \tf |_{\bk} (r,\theta) = |\grad^b f |_{b} (\k r,\theta)$,
\item $\Lap_{\bk} \tf  (r,\theta) = \k \Lap_{b} f(\k r,\theta)$,
\item $ |\hess_{\bk} \tf |_{\bk}(r,\theta) = \k |\hess_{b} f |_{b}(\k r,\theta)$, 
\item $|\hess_{\bk}(\tf)(\grad^{\bk}\tf,\cdot)| (r,\theta)=\k |\hess_{b}(f)(\grad^{b}f,\cdot)| (\k r,\theta)$,
\item $\la \hess_{\tb} (\tf), d\tf \otimes d\tf \ra_{\k} (r,\theta) = \k \la \hess_{b} (f), df \otimes df \ra (\k r,\theta)$,
\item $\la d\tf, d\Vk \ra_{\k}(r,\theta) =\k \la df, d V \ra( \k r, \theta)$,
\item $\la d\Vk, d\Vk \ra_{\k} (r,\theta)= \k^2 \la dV , d V \ra( \k r, \theta)$, and
\item $\la \hess_{\tb}(\tf),d\tf \otimes d\Vk \ra_{\k} (r,\theta) = \k^2 \la \hess_{b}(f),df \otimes dV \ra (\k r,\theta) $.
\end{enumerate}
\end{lemma}

\begin{lemma} \label{scalar-cond-pres}
If $f$ is an asymptotically hyperbolic function with $R(f) \geq -n(n-1)$, then the scalar curvature of the graph of $\tilde f$ in $\Hb_{\k}^{n+1}$ is greater than or equal to $-\k^2 n(n-1)$, that is, $\mathcal{R}_{\tf,\k} \geq 0$, where $\mathcal{R}_{f,\k} = R_{\k}(f) + \k^2 n(n-1)$.
\end{lemma}
\begin{proof}
For simplicity we avoid writing the dependence on $\theta$. From \cite[Eq. (6)]{D-G-S} and using Lemma \ref{formulae}, we see that
\begin{align*}
\begin{split}
&\mathcal{R}_{\tf,\k}(r) \\
=&\dfrac{\Vk(r)^2}{1 + \Vk(r)^2|\grad^{\bk} \tf |^2_{\bk}(r)} \bigg[ (\Lap_{\bk} \tf (r))^2  - |\hess_{\bk} \tf |^2_{\bk}(r)     \\
&+ \dfrac{2 \Vk(r)^2}{1 + \Vk(r)^2 |\grad^{\bk} \tf |^2_{\bk}(r)} \( |\hess_{\bk} \tf (\grad^{\bk} \tf,\cdot )|^2_{\bk}(r) - \Lap_{\bk}\tf(r) \la \hess_{\bk}\tf, d\tf \otimes d\tf \ra_{\k}(r) \) \\
&+ \dfrac{2\la d\tf, d\Vk \ra(r)}{\Vk(r)(1 + \Vk(r)^2 |\grad^{\bk} \tf |^2_{\bk}(r))} \( \Lap_{\bk}\tf(r) - \Vk(r)^2 \la \hess_{\bk}\tf,d\tf \otimes d\tf \ra_{\k}(r)  +\dfrac{\la d\tf, d\Vk \ra_{\k}(r)}{\Vk(r)}\) \\
&+2 \dfrac{\la d\tf, d\Vk \ra(r)}{\Vk(r)}\Lap_{\bk} \tf(r) -\dfrac{2 }{1 + \Vk(r)^2 |\grad^{\bk} \tf |^2_{\bk}(r)}| \grad^{\bk} \tf |^2_{\bk}(r)\bigg\vert \dfrac{d\Vk(r)}{\Vk(r)} \bigg\vert^2 \\
& -\dfrac{4}{1 + \Vk(r)^2 |\grad^{\bk} \tf |^2_{\bk}(r)}\dfrac{1}{\Vk(r)}\la \hess_{\bk} \tf, d\tf \otimes d\Vk \ra (r) \bigg]
\end{split}\\
\begin{split}
=&\dfrac{\k^2 V(\k r)^2}{1 + V(\k r)^2|\grad^{b} f |^2_{b}(\k r)} \bigg[ (\Lap_{b} f (\k r))^2  -  |\hess_{b} f |^2_{b}(\k r)     \\
&+ \dfrac{2 V(\k r)^2}{1 + V(\k r)^2 |\grad^{b} f |^2_{b}(\k r)} \( |\hess_{b} f (\grad^{b} f,\cdot )|^2_{b}(\k r) -  \Lap_{b}f(\k r) \la \hess_{b}f, df \otimes df \ra(\k r) \) \\
&+ \dfrac{2  \left\la  df, dV \right\ra(\k r)}{V(\k r)(1 + V(\k r)^2 |\grad^{b} f |^2_{b}(\k r))} \(  \Lap_{b}f(\k r) - V(\k r)^2 \la \hess_{b}f,df \otimes df \ra(\k r)  +\dfrac{ \left\la df, dV \right\ra(\k r)}{V(\k r)}\) \\
&+2 \dfrac{ \left\la df, dV \right\ra(\k r)}{V(\k r)} \Lap_{b} f(\k r) -\dfrac{2 }{1 + V(\k r)^2 |\grad^{b} f |^2_{b}(\k r)}| \grad^{b} f |^2_{b}(\k r) \bigg\vert \dfrac{dV(\k r)}{V(\k r)} \bigg\vert^2 \\
&-\dfrac{4 }{1 + V(\k r)^2 |\grad^{b} f |^2_{b}(\k r)}  \dfrac{1}{V(\k r)}\la \hess_{b} f, df \otimes dV \ra(\k r) \bigg]
\end{split} \\
&=\k^2 (R(f)(\k r) + n(n-1)).
\end{align*}
\end{proof}

\begin{lemma} \label{mass-1}
 Let $f$ be a balanced asymptotically hyperbolic function of mass $\m=\m_1(f)$ and define the function 
\be \label{f-tilde}
\tf(r,\theta) = \m^{-\frac{1}{n-2}}(f( \m^{\frac{1}{n-2}} r, \theta) - h_0).
\ee 
Then 
\bee
\m_{\m^{\frac{1}{n-2}}}(\tf ) = 1.
\eee
\end{lemma}

\begin{proof}
From Definition \ref{k-mass}, the $\k$-mass of $\tf$ is given by
\begin{align*}
\m_{\k}(\tf) &= \dfrac{1}{c_n} \lm{r}{\infty} \dint_{S_r} \( \Vk\( \dv^{\bk} \ek - d \tr^{\bk} e \) + \( \tr^{\bk} \ek \) - \ek\(\grad^{\bk} \Vk, \cdot\)   \) \( \nu_r \) d\mu^{\bk}. 
\end{align*}

Direct calculations as in \cite[Section 3]{D-G-S} combined with Lemma \ref{formulae} give a way to write the integrand (dropping the dependence on $\theta$ as before) as follows: 
\begin{align*}
&\Vk(\dv^{\bk} \ek(\pr_r) - d \tr^{\bk} \ek(\pr_r)) + (\tr^{\bk} \ek) d(\Vk)(\pr_r) - \ek(\grad^{\bk} \Vk, \pr_r) ) \\
\begin{split}
=&V(\k r)^3 (\Lap_{\bk} \tf)(r)  d\tf(\pr_r) - V(\k r)^3\la \hess_{\bk}(\tf)(r) , d\tf (r) \otimes \, \pr_r \, \ra \\
&\quad  + V(\k r)^2\la d (V(\k r)), d\tf \ra  d\tf(\pr_r) - V(\k r)^2 |d\tf|_{b_{\k}}^2(r) d(V(\k r))(\pr_r) 
\end{split}\\
\begin{split}
=&\k V(\k r)^3  (\Lap_b f)(\k r)  df(\pr_r)(\k r) - \k V(\k r)^3\la \hess_{b}(f)(\k r) , df (\k r)  \otimes \, \pr_r \, \ra \\
&\quad  + \k V(\k r)^2\la dV(\k r), df(\k r) \ra df(\k r)(\pr_r) - \k V(\k r)^2  |df|^2_b(\k r)  dV(\k r)(\pr_r) 
\end{split}\\
\begin{split}
&= \k \( V\( \dv^{b} e - d \tr^{b} e \) + \( \tr^{b} e \) - e\(\grad^{b} V, \cdot\)   \) (\k r), 
\end{split}
\end{align*}
where in the last line we mean that the whole expression in parenthesis is being evaluated at $\k r$. Hence, setting $\k = \m^{\frac{1}{n-2}}$,
\begin{align*}
\m_{\k}(\tf) &= \dfrac{1}{c_n} \lm{r}{\infty} \dint_{S_r} \( \Vk\( \dv^{\bk} \ek - d \tr^{\bk} e \) + \( \tr^{\bk} \ek \) - \ek\(\grad^{\bk} \Vk, \cdot\)   \) \( \nu_r \) d\mu^{\bk} \\
&= \dfrac{1}{c_n} \lm{r}{\infty} \dint_{S_r} \k \(\( V\( \dv^{b} e - d \tr^{b} e \) + \( \tr^{b} e \) - e\(\grad^{b} V, \cdot\)   \) (\k r) \)\( \nu_r \) \dfrac{\sinh(\k r)^{n-1}}{\k^{n-1}} dS_{n-1} \\
&= \dfrac{1}{c_n} \lm{\tilde r}{\infty} \dint_{S_{\frac{\tilde r}{\k}}} \k \(\( V\( \dv^{b} e - d \tr^{b} e \) + \( \tr^{b} e \) - e\(\grad^{b} V, \cdot\) \) (\tilde r) \) \( \nu_r \) \dfrac{V(\tilde r)^{n-1}}{\k^{n-1}} dS_{n-1} \\
&=\dfrac{\m}{\k^{n-2}}, \\
\end{align*}
that is,
\bee
\m_{\m^{\frac{1}{n-2}}}(\tf) = 1.
\eee
\end{proof}

The following lemma expresses the volume of a level set of $\tf$ with respect to $b_{\k}$  (as above) in terms of the volume  of a level set of $f$ with respect to $b$. This relation will be particularly useful to derive the differential inequality in Lemma \ref{f-h0-estimate} below.

\begin{lemma} \label{vol-rel}
Suppose that the level set of $\tf$, $\tilde{\S}_h = \{ \tf = h \}$, is smooth and  star-shaped. Then 
\be
\tilde \V_{\k}(h) = \dfrac{1}{\k^{n-1} }\V_1(h_0 + \k h), 
\ee
where $\tilde \V_{\k}(h)$ denotes the volume of $\tilde{\S}_h$ with respect to $b_{\k}$.
\end{lemma}
\begin{proof}

We first note that $\tilde{\S}_h = \{ f(\k r,\theta) = \k h + h_0 \}$. Let $\S$ be the surface obtained by rescaling the surface $\S_{\k h + h_0}= \{ f = \k h + h_0 \}$ by $\k^{-1}$. By hypothesis, $\S$ is star-shaped and hence we can consider a parametrization $\Phi: \bS^{n-1} \longrightarrow \S \subset \R^n$ given by $\Phi(\theta) = (\phi(\theta),\theta)$, where $\phi$ is a smooth positive function on $\bS^{n-1}$. Then, clearly $\tf (\Phi(\theta))= h$.  Let $\g^{\k}$ denote the metric $b_{\k}$ restricted to $\S_h$, which is given by

\begin{align*}
\g^{k}_{ij}(\theta) &= \phi_i(\theta)\phi_j(\theta) + \dfrac{\sinh(\k \phi(\theta))^2}{\k^2}\s_{ij} \\
&=\dfrac{1}{\k^2} \left[ \k^2 \phi_i(\theta)\phi_j(\theta) + \sinh(\k \phi(\theta))^2\s_{ij} \right],
\end{align*}
where the latin indices denote derivatives on $\bS^{n-1}$.
On the other hand, let $\g$ be the metric $b$ restricted to $\S_{\k h + h_0}$, that is, the surface obtained by scaling $\S$ by $\k$, given by
\begin{align*}
\g_{ij}(\theta) =  \k^2 \phi_i(\theta) \phi_j(\theta) + \sinh(\phi(\k \theta))^2 \s_{ij}.
\end{align*}
We therefore have
\begin{align*}
\tilde {\V}_{\k}(h) &= \dint_{\bS^{n-1}} \sqrt{ \det{\g^{\k}}}\, d\theta_1 \wedge \hdots \wedge d\theta_{n-1} \\
&=\dfrac{1}{\k^{n-1}} \dint_{\bS^{n-1}} \sqrt{ \det{\g}}\, d\theta_1 \wedge \hdots \wedge d\theta_{n-1} \\
&=\dfrac{1}{\k^{n-1}}  \V_1(h_0+\k h).
\end{align*}
\end{proof}

We now recall the following Lemma from \cite{H-L}, which will be applied in the proof of Lemma \ref{f-h0-estimate}.

\begin{lemma}[\cite{H-L}] \label{diff-ineq-lemma}
Let $V:[a,b] \To \R$ be non-decreasing. Suppose that $V' \geq F(V)$ holds almost everywhere in $[a,b]$. Suppose that $F$ is non-decreasing and continuously differentiable. Let $Y$ be a $C^2$ function satisfying 
\bee
Y' = F(Y) \ \ \text{and} \  \ Y(a) \leq V(a).
\eee
Then $Y \leq V$ on $[a,b]$.
\end{lemma}

We are ready to estimate the height of an asymptotically hyperbolic graph in terms of its mass.

\begin{lemma} \label{f-h0-estimate}
Suppose that $f$ is a balanced asymptotically hyperbolic function, whose graph has scalar curvature $R(f) \geq -n(n-1)$ in $\Hb^{n+1}$ $(n \geq 3)$ and mass $\m := \m(f)$. Suppose $f$ has upward pointing mean curvature vector field and that the level sets of $f$, $\S_h$,  are star-shaped and outer-minimizing for almost every $h$ in the range of $f$. Then there exists a constant $C=C(n)$, such that 
\be
0 < \sup(f) - h_0 < C\m^{\frac{1}{n-2}},
\ee
where $h_0$ is defined in \eqref{h0}.
\end{lemma}
\begin{proof}
Define $\tilde f$ as in \eqref{f-tilde}. Denote by $\tilde\V(h)$ the volume of the level set $\tilde\S_h = \{ \tilde f = h \}$ with respect to the metric $b_{\k}$ with $\k =\m^{\frac{1}{n-2}}$. (Recall that by Lemma \ref{mass-1}, with this choice of $\k$, we have $\m_{\k}(\tf) = 1$.) Then we can define $\tilde \V(0) =  \lm{h}{0} \tilde \V_{\m^{\frac{1}{n-2}}} (h) $ and hence
\begin{align*}
\tilde \V (0) = \tilde \V_{\m^{\frac{1}{n-2}}}( 0 ) &= \lm{h}{0} \tilde \V_{\m^{\frac{1}{n-2}}} (h) \\
&= \lm{h}{h_0^+} \m^{-\frac{n-1}{n-2}}\V_1(h) \\
&> 2\beta \m^{-\frac{n-1}{n-2}} \m^{\frac{n-1}{n-2}} \w_{n-1} \\
&=2 \beta \w_{n-1},
\end{align*}
where we have used Lemma \ref{vol-rel}. Therefore, by Lemma \ref{scalar-cond-pres}, we can apply Lemma \ref{vol-k-estimate} to obtain the differential inequality
\begin{align*}
\tilde \V'(h) \geq c_n\dfrac{2}{3\sqrt{3}}\(\dfrac{1}{2 \w_{n-1}}\tilde \V(h) -1 \)^{3/2}
\end{align*}
with
$\tilde \V( 0 ) \geq  2 \beta \w_{n-1}  $. Note that $\frac{1}{2 \w_{n-1}}\tilde \V(h)  > \beta > 1$, since $\V_1(h + h_0) > 2\beta \m^{\frac{n-1}{m-2}}\w_{n-1}$.

In order to apply Lemma \ref{diff-ineq-lemma}, define $Y$ to be the unique solution to
\be
\left\{
\begin{split}
Y'(h)&=c_n \dfrac{2}{3\sqrt{3}} \(\dfrac{1}{2 \w_{n-1}}Y(h) -1 \)^{3/2} \\
Y(0) &= 2 \beta \omega_{n-1}
\end{split} \right.
\ee
Then, by Lemma \ref{diff-ineq-lemma}, $Y(h) \leq \tilde \V(h)$. Thus for $n \geq 3$, $Y(h)$ must tend to infinity at a finite height $C=C(n)$. Then $\tilde\V$ also tends to infinity at a finite height $\tilde h_{\max} \leq C$. From Lemma \ref{lemma-height}, it follows that $0 < \sup \tf < C$, then $0 < \sup f - h_0 < C\m^{\frac{1}{n-2}}$.

\end{proof}

\section{Convergence in the Flat Norm} \label{convergence}

The flat convergence has been thoroughly studied and we refer the reader to \cite{Sormani} for a good general description and to find the precise references for a more complete exposition of this topic. For the reader's convenience we include some basic definitions and facts.

Given a Riemannian manifold $M$, recall that a submanifold $N$ can be seen as an integral current $T$ (of multiplicity one) and then its boundary $\pr T$ is $\pr N$ viewed as an integral current. In this case, the mass of the current $\bM(T)$ is simply the volume of the submanifold $N$.

\begin{definition}
Let $U$ be an open set in $\Hb^{n+1}$, and let $T_1$ and $T_2$ be integral $k$-currents in $\Hb^{n+1}$. Let $\bM_U$ denote the mass of a current in $U$. The flat distance between $T_1$ and $T_2$ in $U$ is defined as
\bee
d_{\mathcal{F}}(T_1,T_2) = \inf \{ \bM_U(A) + \bM_U(B) \, : \, T_1-T_2=A+\pr B  \},
\eee
where the infimum is taken over all integral $k$-currents $A$, and all $(k+1)$-currents $B$ in $\Hb^n$.
\end{definition}

We now estimate the flat distance of the graph of an asymptotically hyperbolic function $f$ with $\{ s = h_0 \}$, which is isometric to $\Hb^n$. We will follow the idea from \cite{H-L}. An $\Hb$-ball of radius $\rho$ is the set of points in $\Hb^{n+1}$ whose distance to the origin is less or equal to $\rho$. In terms of our coordinate system it can be described as the set of coordinate points $(s,r,\theta)$ such that $\cosh^2(r)s^2 + \sinh^2(r) \leq \rho^2$, together with the origin.

\begin{thm} \label{F-distance}
For $n \geq 3$, let $U$ be an $\Hb$-ball of radius $\rho$ in $\Hb^{n+1}$. Let $f$ be a balanced asymptotically hyperbolic function of mass $\m = \m(f)$, whose graph has scalar curvature $R(f) \geq -n(n-1)$ and mean curvature vector pointing upward. Suppose that the level sets $\S_h$ of $f$ are star-shaped and outer-minimizing for almost every $h$. Then
\bee
d_{\mathcal{F}_U}(\graph[f],\{ s=h_0 \}) \leq \tilde{c}_n [k_1 \m + k_2 \m^{\frac{1}{n-2}}],
\eee
where $\tilde{c}_n$ is a constant depending only on $n$ and $k_1$ and $k_2$ are constants depending only on $\rho$.
\end{thm}

\begin{proof}
Denote by $\bar f$ the extension of $f$ obtained by setting $f$ to be constant on  $\bO$ and $\bar f = f$ on $\Hb^n \setminus \Omega$. We want to find currents $A$ and $B$ such that $\graph[f]-\{s=h_0\} = A +  \pr B$ in $U$, where both graphs are taken with the upward orientation. As in \cite{H-L}, define $A$ to be the region bounded by $\pr (\graph[f])$ (which lie in a horizontal plane) taken with the downward orientation, so that $\graph[f] - A$ is $\graph[\bar f]$, i.e., $f$ with the boundary filled in. Let $B$ be the region of $\Hb^{n+1}$ under $\graph[f] - A$  minus the region of $\Hb^{n+1}$ under $\{ s=h_0\}$, both taken with positive orientation. Then $(\graph[f] - A) - \{ s=h_0 \} =\pr B$. Now, $B=B_+ + B_-$, where $B_+$ is the region of $\Hb^{n+1}$ below $\graph[f] - A$ and above $\{ s=h_0 \}$ with positive orientation, and $B_-$ is the region of $\Hb^{n+1}$ with negative orientation above $\graph[f]-A$ and below $\{ s=h_0 \}$ (see Figure \ref{flat-dist}). Since they are disjoint, it follows that $\bM_U(B) = \bM_U(B_+) + \bM_U(B_-)$.

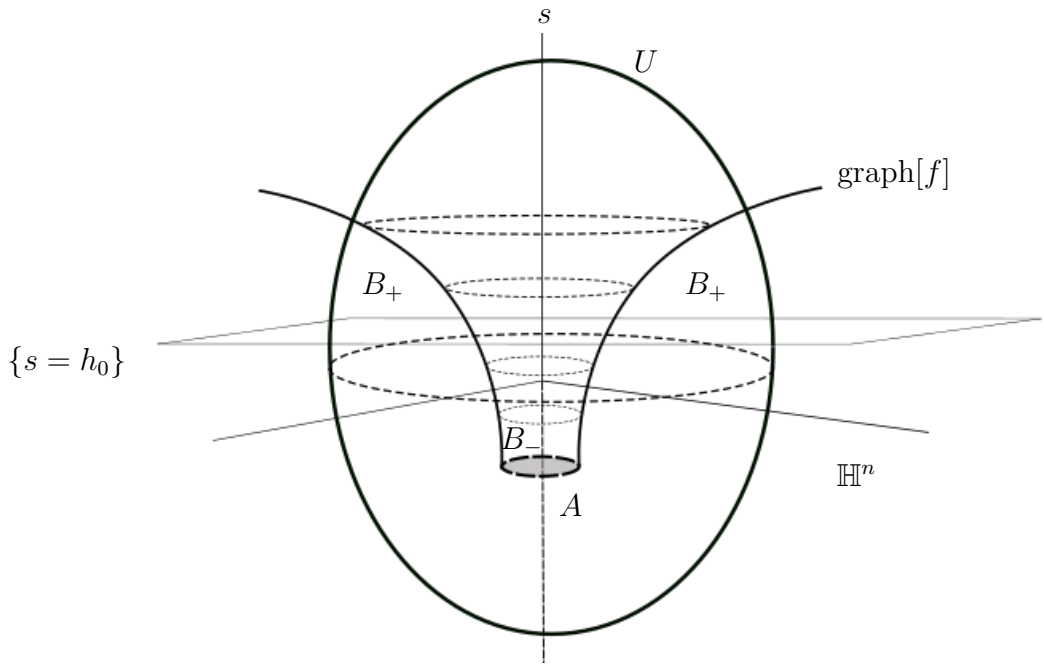
\begin{figure}[ht!]
\begin{center}
\begin{tikzpicture}
    \node[anchor=south west,inner sep=0] at (0,0) {\includegraphics[scale=.8]{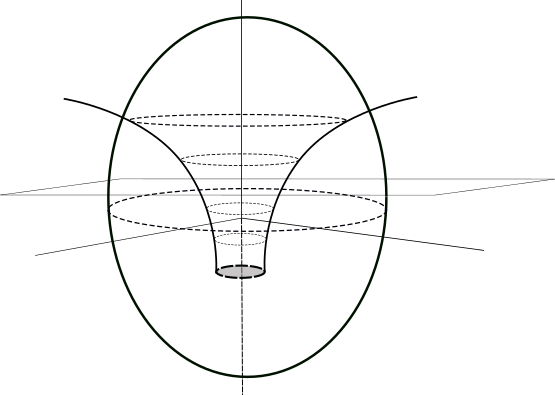}};
    \node at (-1.2,4) {$\{ s = h_0 \}$};
    \node at (5.5,2.1) {$A$};
    \node at (3,5) {$B_+$};
    \node at (7.3,5) {$B_+$};
    \node at (4.85,2.95) {$B_-$};
    \node at (6.5,8) {$U$};
    \node at (9.8,6.5) {$\graph[f]$};
    \node at (9.3,2.5) {$\Hb^n$};
    \node at (5.15,8.6) {$s$};
\end{tikzpicture}
\end{center}
\caption[]{Choice of the sets $A$ and $B$ to estimate the flat distance between $\graph[f]$ and $\{ s = h_0 \}$. } \label{flat-dist}
\end{figure}

Note that since the level sets are assumed to be mean convex for almost every height, when the boundary $\pr \Omega$ is non-empty, it is mean convex in $\Hb^n$. By the isoperimetric inequality in $\Hb^n$ \cite{Yau} combined with the Riemannian Penrose-like inequality \eqref{RPI-AH} from \cite{D-G-S}, we have
\bee
\bM_U(A) = \vol(A \cap U) \leq \tilde{c}_n | \pr \Omega | \leq  \tilde{c}_n \m,
\eee
where $\tilde{c}_n$ is a constant depending only on $n$. From Lemma \ref{f-h0-estimate},
\bee
\bM_U(B_+) = \vol(B_+ \cap U) \leq C \cosh(\rho) \sinh(\rho)^{n-1} \rho \m^{\frac{1}{n-2}}.
\eee
For each $h \leq h_0$ (when this set is non-empty), using again the isoperimetric inequality in $\Hb^n$, the fact that $\V_1(h)$ is non-decreasing and the definition of $h_0$, we can estimate the volume on each slice by
\bee
\vol(B_- \cap U \cap \{ s=h \} ) \leq \vol(\Omega_h ) \leq \tilde{c}(n) \V_1(\S_h) \leq  \tilde{c}(n) \V_1(\S_{h_0}) \leq \tilde{c}_n \m.
\eee
Integrating the over the region $s \in [-\rho,\rho]$, the above implies that
\bee
\bM_U( B_- ) \leq \tilde{c}_n \cosh(\rho)\rho \m,
\eee
for some constant $\tilde{c}_n$ depending only on $n$. Therefore,
\bee
d_{F_U}(\graph[f],\{ s=h_0 \}) \leq \tilde{c}_n[(\cosh(\rho)\rho + 1)\m + \cosh(\rho)\sinh(\rho)^{n-1}\rho \m^{\frac{1}{n-2}}].
\eee
\end{proof}

We now show the convergence of a sequence of balanced asymptotically hyperbolic graphs (properly normalized) to the hyperbolic space $\Hb^n$, in the sense of currents. Clearly, since $M^i = \graph[f_i]$ and $\m^i = \m(f_i)$, the proof of Theorem \ref{main} follows.

\begin{thm}
Let $f_i$ be a sequence of balanced asymptotically hyperbolic graphs with scalar curvature  $R(f_i) \geq -n(n-1)$ and mean curvature vector pointing upward. In addition, suppose that almost every level set of each $f_i$ is star-shaped and outer-minimizing in $\Hb^n$. Normalize the height of $f_i$ so that $h_0=0$. Then if  $\m(f_i) \to 0$, then the sequence $\{ f_i \}$ converges to $\{ s=0 \}$ in the sense of currents.
\end{thm}

\begin{proof}
The proof is identical to \cite[Theorem 1.1]{H-L} and it is included for completeness. Let $\w$ be a compactly supported $n$-form in $\Hb^{n+1}$. Then there exists a $\rho > 0$ such that the support of $\w$ is contained in an $\Hb$-ball $U$. Let $\Pi = \{ s=0 \}$. We have
\be \label{curr-conv}
(\graph[f_i] - \Pi)(\w) = A_i(\w) + B_i(d\w),
\ee
with $A_i$ and $B_i$ as in the proof of Theorem \ref{F-distance}. Since the right hand side of \eqref{curr-conv} is controlled by the flat distance of $\graph[f_i]$ and $\{ s = 0 \}$, by Theorem \ref{F-distance}, $(\graph[f_i] - \Pi)(\w)  \to 0$ as $\m(f_i) \to 0$.
\end{proof}

\bibliographystyle{amsplain}
\bibliography{PMTAH}
\vfill

\end{document}